\documentclass[a4paper, leqno, myheadings, 11pt, twoside]{article}
\usepackage{mathrsfs}

\usepackage{amsmath, amsthm, amssymb, amscd, amsxtra,graphicx}
\usepackage{latexsym, amsfonts,indentfirst}

\def\rdimen{$r-$di\-men\-sional }

\def\mo{M\"obius }

\def\hatc{\hat{\mathbb{C}}}

\def\rn{{\hat{\mathbb{R}}^n}}

\def\rnn{\mathbb{R}^n}
\def\hn{\mathbb{H}^n}

 \makeatletter
\@addtoreset{equation}{section}

\newtheorem{Thm}{Theorem}

\newtheorem{lem}{Lemma}

\newtheorem{theo}{Theorem}

\newtheorem{conj}{Conjecture}
\newtheorem{open}{Problem}

\theoremstyle{remark}
\newtheorem{rem}{\textbf{Remark}}

\begin{document}

\title{\bf{Fundamental theorem of hyperbolic geometry without the injectivity assumption}}

\author{ GUOWU YAO\\
  Department of Mathematical Sciences,
  Tsinghua University\\Beijing, 100084, P.R. China\\
e-mail: \texttt{gwyao@math.tsinghua.edu.cn} }

\date{}
\maketitle\renewcommand\abstract{\noindent\large{A}\small{BSTRACT.
\;}}
\begin{abstract}Let $\hn$ be the $n-$dimensional hyperbolic space.
It is well known that, if $f:\;\hn\to \hn$ is a bijection that
preserves $r-$dimensional hyperplanes, then  $f$ is an isometry. In
this paper we make neither  injectivity nor $r-$hyperplane
preserving assumptions on $f$ and prove the following result:
 \\
\indent Suppose that $f:\;\hn\to \hn$ is a surjective map
 and
maps an $r-$hyperplane  into an $r-$hyperplane, then $f$ is an
isometry.

 The Euclidean
version was obtained by A. Chubarev and I. Pinelis in 1999 among
other things. Our proof is essentially different from their and the
similar problem arising in the spherical case is open.

\end{abstract}

 \footnote{The  author was  supported by  a Foundation
for the Author of National Excellent Doctoral Dissertation (Grant
No. 200518) and the National Natural Science Foundation of China.}

\begin{centering}\section{\!\!\!\!\!{. }
Introduction}\label{S:intr}\end{centering}

Let $\rn=\mathbb{R}^n\cup\{\infty\}$ where $\rnn$ is the
$n-$di\-men\-sional Euclidean space and let $\mathbb{H}^n$ be the
$n-$di\-men\-sional hyperbolic space. A map $f$ of  $\rn$ to itself
is called $r-$sphere  preserving  if $f$ maps an \rdimen
 sphere  onto an \rdimen  sphere. Similarly, a
map $f$ of  $\rnn$ (or $\hn$) to $\rnn$ (or $\hn$) is called
$r-$hyperplane preserving  if  $f$ maps an  \rdimen hyperplane onto
an \rdimen hyperplane in $\rnn$ (or $\hn$). In particular, when
$r=1$, we call the corresponding map $f$ to be a circle-preserving
(line-preserving, geodesic-preserving) map in $\rn\,(\rnn,\, \hn)$,
respectively. In the sequel, we prescribe $n\geq2$ and $1\leq r<n$.

The property of a  \mo transformation acting on $\hatc$ is so clear
and the relations between  \mo transformation and some of its
property have been extensively studied.  For examples,
Carath\'eodory first proved that if $f:\,\hatc\to\hatc$ is a
circle-preserving bijection, then $f$ is a \mo transformation (see
\cite{Car} or \cite{Sch});
 Nehari
\cite{Ne} showed that if $f:\,\hatc\to\hatc$ is a non-constant
meromorphic function that preserves circles, then $f$ is a \mo
transformation.

Of course, the analogous problem for affine (or isometric)
transformations on $\rnn$ (or $\hn$) is also concerned. In
 \cite{Je}, Jeffers obtained the following extension of Carath\'eodory's result to all three cases (for concision, we combine
three theorems obtained by him into one).
\begin{theo}\label{Th:AM}
Suppose that $f:\,\rn\to\rn$ ($\rnn\to\rnn$, $\hn\to \hn$) is a
bijection that preserves \rdimen spheres (\rdimen hyperplanes). Then
$f$ is a(n) \mo (affine, isometric) transformation.
\end{theo}

An $r-$sphere preserving  map $f$ is called degenerate if its image
$f(\rn)$ is an \rdimen  sphere; otherwise, $f$ is called
non-degenerate. The reader will easily guess  the proper definitions
for non-degenerate and degenerate maps in the Euclidean and
hyperbolic settings. In a recent article \cite{LW}, B. Li and Y.
Wang made neither injectivity nor surjectivity  assumptions on $f$
and proved
\begin{theo}\label{Th:ly0}
Suppose that $f:\,\rn\to\rn$ ($\rnn\to\rnn$, $\hn\to \hn$)  is a
circle-preserving (line-preserving, geodesic-preserving) map. Then
$f$ is a(n) \mo (affine, isometric)  transformation if and only if
$f$ is non-degenerate.
\end{theo}
The existence of degenerate maps was shown in \cite{LW,Yao}. More
recently, the author joint with B. Li \cite{LY} obtained the
following generalization of Theorem \ref{Th:ly0}.
\begin{theo}\label{Th:lw}Suppose that $f:\,\rn\to\rn$ ($\rnn\to\rnn$, $\hn\to \hn$)
 is an $r-$sphere preserving ($r-$hyperplane preserving) map. Then $f$ is a(n) \mo  (affine, isometric)
 transformation if and
only if $f$ is non-degenerate.
\end{theo}

In \cite{CP}, Chubarev and Pinelis  showed, among other things, that
the injective condition for the Euclidean case $\rnn$ in Theorem
\ref{Th:AM} can be removed. Precisely, the following theorem was
implied.
\begin{theo}\label{Th:lyee}
Suppose that $f:\,\rnn\to\rnn$ is a surjective map  and  maps every
\rdimen hyperplane into an \rdimen  hyperplane. Then $f$ is an
affine transformation.
\end{theo}
Inspired by Theorem \ref{Th:lyee}, the following two conjectures
were naturally posed in \cite{LY}:
\begin{conj}\label{Th:conj1}
Suppose that $f:\,\rn\to\rn$ is a surjective map  and  maps every
\rdimen sphere  into an \rdimen  sphere. Then $f$ is a \mo
transformation.
\end{conj} and
\begin{conj}\label{Th:conj2}
Suppose that $f:\,\hn\to\hn$ is a surjective map  and  maps every
\rdimen hyperplane into an \rdimen hyperplane. Then $f$ is an
isometric transformation. \end{conj}

The aim of  this paper is to prove  Conjecture \ref{Th:conj2} by
applying Theorem \ref{Th:lw} but leave Conjecture \ref{Th:conj1}
open. For completeness, we also give a simple proof of Theorem
\ref{Th:lyee} in Section \ref{S:simple}.

Other results in the line can be found in \cite{AM, BM,
CP,GW,Low,Ste}.

\begin{rem}Recently, the author \cite{Yao1} proved  that Conjecture 1
is true in the case $r=n-1$.\end{rem}

\begin{centering}\section{\!\!\!\!\!{. }      Some preparations}\label{S:reduction}
\end{centering}
This section is devoted to reduce the proof of Conjecture
\ref{Th:conj2} to that of the special case when $r=1$. That is, we
only need to prove that,
\begin{Thm}\label{Th:conjthm2}
Suppose that $f:\,\hn\to\hn$ is a surjective map  and  maps every
geodesic into a geodesic. Then $f$ is an isometric transformation.
\end{Thm}

This reduction clearly depends on the following  lemma.

\begin{lem}\label{Th:lylem11}Suppose (i) there exists some $r$ such
that the map $f:\hn\to\hn$ maps every  \rdimen hyperplane into an
\rdimen hyperplane, (ii) $f(\hn)$ is not contained in an \rdimen
hyperplane. Then for any given $k$-di\-men\-sional hyperplane
$\Gamma\subseteq \hn$ ($1\leq k\leq r$), $f$ maps $\Gamma$ into a
$k$-di\-men\-sional hyperplane. In particular, $f$ maps a geodesic
into a geodesic.
\end{lem}
Throughout our discussion, lower case letters will denote points,
upper case letters sets of points, subscripts for like objects, and
primes for images under the map $f$. The notable exception to these
conventions will be when the image $f(\Lambda)$ of a set $\Lambda$
is not presumed to be $\Lambda'$ but  we will write
$f(\Lambda)\subseteq \Lambda'$.

For a nonempty subset $A$ with $\#A\geq2$ in $\hn$, let $\prod A$
denote the $t-$di\-men\-sional hyperplane containing $A$ such that
$t$ is the smallest positive integer. It is easy to see that $\prod
A$ and $t$ are uniquely determined by the set $A$.

 Now, we prove Lemma \ref{Th:lylem11}:

If $r=1$, it is a fortiori.  Let $r\geq 2$ and  $k=r-1$. Embedding
$S$ into some \rdimen hyperplane $\Gamma$, we have
$\Gamma'=f(\Gamma)$ as an $r$-di\-men\-sional hyperplane by
hypothesis. Since $f(\hn)$ is not contained in an \rdimen
hyperplane, we can find a point $p\in \hn\backslash \Gamma$ such
that $p'=f(p)\not\in \Gamma'$.

Letting $\Gamma_1=\prod\{S,p\}$, then $\Gamma_1$  is an \rdimen
hyperplane. Set $\Gamma'_1=f(\Gamma_1)$. Since
\[f(S)=f(\Gamma\cap\Gamma_1)\subseteq f(\Gamma)\cap
f(\Gamma_1)\] and $f(\Gamma)\cap f(\Gamma_1)=\Gamma'\cap \Gamma'_1$
is contained in  an $(r-1)$-di\-men\-sional hyperplane, the lemma
holds for $k=r-1$.  It is clear that $f(\hn)$ is also not contained
in an $(r-1)-$di\-men\-sional  hyperplane. Thus, we can inductively
backward prove that if this lemma holds for $k$ ($\geq2$), then it
does for $k-1$. This lemma then follows.

\begin{centering}\section{\!\!\!\!\!{. }     Proof of Theorem \ref{Th:conjthm2}}\label{S:conj2}
\end{centering}
In this section, we  prove Theorem \ref{Th:conjthm2}. Throughout
this section except in Lemma \ref{Th:thm2lem1}, we assume that $f$
satisfies the conditions of Theorem \ref{Th:conjthm2}. $l_{xy}$
always denotes the geodesic determined by $x$ and $y$ in $\hn$.

\begin{lem}\label{Th:thm2lem1}Suppose $f:\,\hn\to\hn$ maps a geodesic into a
geodesic. Then $f$ maps an \rdimen hyperplane into an \rdimen
hyperplane for  $1\leq r\leq n-1$.
\end{lem}
\begin{proof}
We use induction. Let $n\geq 3$ and assume that $f$ maps an \rdimen
hyperplane into an \rdimen hyperplane for some $r\in [1,n-2]$. We
need to show that
 $f$ maps an $(r+1)-$di\-men\-sional
hyperplane into an $(r+1)-$di\-men\-sional hyperplane.

Suppose not. Then there exists an $(r+1)-$di\-men\-sional hyperplane
$S$ such that $\prod f(S)$ has dimension   $d\geq r+2$. Therefore,
there exist $r+3$ points $\{p'_1,p'_2,\cdots,p'_{r+3}\}$  in $f(S)$
such that no $(r+1)-$di\-men\-sional hyperplane contains them  and
the
 hyperplane $\prod K'$ spanned by $K'=\{p'_1,p'_2,\cdots,p'_{r+3}\}$
  has dimension $r+2$. On the
other hand, there exist $r+3$ distinct points
$\{p_1,p_2,\cdots,p_{r+3}\}$ in $S$ such that $f(p_i)=p'_i$
($i=1,2,\cdots,r+3$).

It is clear that no  \rdimen hyperplane contains
 more than $r+2$ points of $\{p_1,p_2,\-\cdots,\-p_{r+3}\}$ by the  inductive assumption.
 Therefore, every $r+1$ points of
$\{p_1,p_2,\cdots,p_{r+2}\}$ can span a unique $r$-di\-men\-sional
hyperplane and these $r+2$ spanned hyperplanes divide the
$(r+1)-$di\-men\-sional hyperplane $S$ into $2^{r+2}-1$ disjoint
parts. The point $p_{r+3}$ is located inside some part.

Observe that $\{p_1,p_2,\cdots,p_{r+2}\}$ frames an $(r+1)$-simplex
in $S$. Anyway there exists at least a point of
$\{p_1,p_2,\cdots,p_{r+2}\}$, say $p_1$, such that the geodesic
 $l_{p_1p_{r+3}}$ crosses the \rdimen hyperplane
 $\Lambda=\prod\{p_2,p_3,\cdots,p_{r+2}\}$.  Letting $q=l_{p_1p_{r+3}}\cap \Lambda$, then
 $q\neq p_1$ and $q'=f(q)\in  \prod\{p'_2,p'_3,\cdots,p'_{r+2}\}$
 by the  inductive assumption. Thus,
 $p'_{r+3}=f(p_{r+3})\in f(l_{qp_1})$  which shows that $p'_{r+3}\in
 \prod\{p'_1,p'_2,\cdots,p'_{r+2}\}$ since $f(l_{qp_1})\subseteq l_{q'p'_1}$. This further indicates that
 $\{p'_1,p'_2,\cdots,p'_{r+3}\}$ is contained in the $(r+1)-$di\-men\-sional hyperplane $\prod
 \{p'_1,p'_2,\cdots,p'_{r+2}\}$, a contradiction. The inductive
 proof  is completed.

\end{proof}

\begin{rem}Lemma \ref{Th:thm2lem1} can be regarded as a converse of
Lemma \ref{Th:lylem11}. We have an essential difficulty in obtaining
its spherical version which is also the only bug to solve Conjecture
1 while we prove Conjecture 1 when $r=n-1$ in \cite{Yao1}. In other
words, the answer to the following problem is crucial to the
solution of  Conjecture 1.\end{rem}

\begin{open} Suppose that $f:\,\rn\to\rn$
($n\geq3$) is a surjective map  and   maps a circle into a circle.
Can we say that $f$ maps an $(n-1)-$dimensional sphere into an
$(n-1)-$dimensional sphere?
\end{open}

\begin{lem}\label{Th:thm2lem2}
Suppose $D$ is a domain in $\hn$. If  $f(D)$ is  contained in an
$(n-1)-$di\-men\-sional hyperplane, then $f$ is constant on $D$.
\end{lem}
\begin{proof}
Suppose not. Then $f(D)$ is  contained in an $(n-1)-$di\-men\-sional
hyperplane, say $\Gamma'\subseteq \hn$, and $f(D)$ contains at least
two points. Let $S=\{w\in \hn:\,f(w)\in \Gamma'\}$.  Obviously,
$D\subseteq S$, $f(S)\subseteq \Gamma'$ and $S\neq \hn$.

\textbf{Claim 1.} $S$ is path-connected.

We may choose two points $p,\,q$ in $S$ such that $f(p)\neq f(q)$
since $f(S)=\Gamma'$. Now, for any other point $w\in S$, it is no
harm to assume that $f(w)\neq f(q)$. Thus, the geodesic
$l_{wq}\subseteq S$ since $f(l_{wq})\subseteq \Gamma'$ which implies
that $S$ is path-connected.

\textbf{Claim 2.} $\hn-S$ contains no interior points.

Suppose to the contrary.  Let $p$ be an interior point of $\hn-S$
and $P$ be the largest connected open set in $\hn\backslash S$ such
that $p\in P$. Whence, every set $f(l_{pq}\cap D)$ is a singleton
since otherwise $p'\in \Gamma'$, where the geodesic $l_{pq}$ passes
through $p$ and a point $q\in D$.

Since $f(D)$ contains at least two points, there are  two points,
say $u$ and $v$, such that the geodesics $l_{p'u'}$ and $l_{p'v'}$
are distinct. Therefore, there exists at least a point $x$ in $D$
and a sequence of points $\{x_n\}_{n=1}^{\infty}$ in $D$ such that
\[
\lim_{n\to \infty}x_n=x   \text{  and  } l_{p'x'_n}\neq l_{p'x'},
\;\forall\;n.
\]

Observe that  $f(l_{px_n}\cap D)=x'_n$ and  $f(l_{px}\cap D)=x'$ and
 $x'_n\neq x'$ for all $n$. We then may choose sufficiently
large integer $m$ and a point $y_m$ in
 $l_{px_m}\cap D$  such that the
geodesic $l_{xy_m}$ through $x$ and $y_m$ crosses the domain $P$.
Recalling  that $f(x)=x'$ and $f(y_m)=x'_m$, we can find a point
$z\in l_{xy_m}\cap P$ such that $f(z)\in f(l_{xy_m})\subseteq
\Gamma'$. This indicates that $z\in S$, a contradiction.

We continue to derive a new contradiction from the above two
\textbf{Claim}s
 as follows.

 Let $\Lambda'\subseteq \hn\backslash\Gamma'$ be an $(n-1)-$di\-men\-sional
 hyperplane. Choose $n$ points $\{p'_1,p'_2,\cdots,p'_n\}$ in $\Lambda'$ such that these $n$ points
 are not contained in an $(n-2)-$ di\-men\-sional hyperplane (when $n=2$, such choice is trivial).
There exist $n$ distinct points $\{p_1,p_2,\cdots,p_n\}$ in $\hn$
such that $f(p_i)=p'_i$ ($i=1,2,\cdots,n$). Let
$\Lambda=\prod\{p_1,p_2,\cdots,p_n\}$ be the hyperplane spanned by
$\{p_1,p_2,\cdots,p_n\}$. It is easy to deduce  from Lemma
\ref{Th:thm2lem1} that the dimension $dim(\Lambda)$ of  $\Lambda$ is
just $n-1$ and $f(\Lambda)\subseteq\Lambda'$. Notice that $\Lambda$
divides $\hn$ into two disjoint domains. Necessarily, $\Lambda\cap
S\neq \emptyset$ by \textbf{Claim}s 1 and 2. Thus, we have
$f(\Lambda\cap S)\neq \emptyset$ which contradicts that
$f(\Lambda)\cap f(S)\subseteq \Lambda'\cap\Gamma'=\emptyset$. This
completes the proof of Lemma \ref{Th:thm2lem2}.
\end{proof}

\begin{lem}\label{Th:thm2lem3}
Suppose $D$ is a domain in $\hn$. Then $f(D)$ cannot be contained in
an $(n-1)-$di\-men\-sional hyperplane  in $\hn$.
\end{lem}
\begin{proof}
Suppose not. Then by Lemma \ref{Th:thm2lem2}, $f$ is constant on
$D$, in other words,  $f$ maps $D$ to a point, say $p'\in\hn$.   Let
$\mathscr{D}$ be the largest connected open set of $\hn$ such that
$f(\mathscr{D})=\{p'\}$. Set $S=\{w\in \hn:\,f(w)=p'\}$.  Obviously,
$D\subseteq\mathscr{D}\subseteq S$ and  $S\neq \hn$.

\textbf{Claim.} $\hn-S$ contains interior points.

We may choose two $(n-1)-$di\-men\-sional hyperplanes $\Phi'$ and
$\Psi'$ in $\hn\backslash\{p'\}$ such that, (i) $\Phi'\cap
\Psi'=\emptyset$ and (ii) the convex set $K'=\{z'\in\hn:
\exists\;x'\in\Phi',\,y'\in \Psi',\, s.t.\;z'\in l_{x'y'}\}$ does
not contain $p'$. By virtue of Lemma \ref{Th:thm2lem1}, one can find
two hyperplanes $\Phi$ and $\Psi$ in $\hn$ such that
$f(\Phi)\subseteq \Phi'$, $f(\Psi)\subseteq \Psi'$ and
$dim(\Phi)=dim(\Psi)=n-1$. It is evident that
$\Phi\cap\Psi=\emptyset$. It is also clear that the convex set
$K=\{z\in\hn: \exists\;x\in\Phi,\,y\in \Psi,\, s.t.\;z\in l_{xy}\}$
contains interior points and $f(K)\subseteq K'$. This claim follows
immediately.

 Now, in virtue of the above \textbf{Claim}, it is easy to choose an interior
point  $q$  of $\hn\backslash S$ and a point $b$ on the boundary of
 $\mathscr{D}$
such that  geodesics through $p$ and points of $\mathscr{D}$ contain
a neighborhood $N$ of $b$. Noticing that all such geodesics are
mapped into $l_{q'p'}$, we have $N\subseteq S$ by Lemma
\ref{Th:thm2lem2}. This contradiction establishes this lemma.
\end{proof}

\begin{lem}\label{Th:thm2lem4}
$f$ maps every $(n-1)-$di\-men\-sional hyperplane in $\hn$ onto an
$(n-1)-$di\-men\-sional hyperplane  in $\hn$, i.e., $f$ is
$(n-1)-$hyperplane preserving.\end{lem}
\begin{proof}
Given an $(n-1)-$di\-men\-sional hyperplane  $S$ in $\hn$, by Lemma
\ref{Th:thm2lem1}, there exists an $(n-1)-$di\-men\-sional
hyperplane $S'\supset f(S)$. We now show that $S'=f(S)$. Suppose
not, then there should exist some  point $a'\in S'\backslash f(S)$.
Let $a\in \hn\backslash S$ be an inverse image of $a'$ under $f$.
The collection of geodesics through $a$ and points of $S$ covers a
domain in $\hn$ which is mapped into the $(n-1)-$di\-men\-sional
hyperplane $S'$. It derives a desired contradiction from Lemma
\ref{Th:thm2lem3}. Thus, we prove that $f$ is $(n-1)-$hyperplane
preserving.
\end{proof}

Finally,  the  proof of Theorem \ref{Th:conjthm2} is concluded by
Lemma \ref{Th:thm2lem4} and Theorem \ref{Th:lw} (let $r=n-1$).

\begin{centering}\section{\!\!\!\!\!{. }
A simple proof of Theorem \ref{Th:lyee}}\label{S:simple}
\end{centering}
By the foregoing reasoning,  the proof of Theorem \ref{Th:lyee}
reduces to that of the following theorem.
\begin{Thm}\label{Th:thm3}
Suppose that $f:\,\rnn\to\rnn$ is a surjective map  and  maps every
line  into a line. Then $f$ is an affine transformation.
\end{Thm}
\begin{proof}

For one thing, it is easy to prove that $f$ maps an
$(n-1)$-di\-men\-sional hyperplane into an $(n-1)-$di\-men\-sional
hyperplane as proving Lemma \ref{Th:thm2lem1}.

We claim that $f$ also maps an $(n-1)$-di\-men\-sional hyperplane
onto an $(n-1)-$di\-men\-sional hyperplane. Suppose to the contrary.
Then there exists an $(n-1)-$di\-men\-sional hyperplane $\Gamma$ in
$\rnn$ such that $f(\Gamma)$ is contained in an
$(n-1)-$di\-men\-sional hyperplane $\Gamma'$ and $\Gamma'\backslash
f(\Gamma)\neq \emptyset.$

Let $p'\in \Gamma'\backslash f(\Gamma)$ and $p$ an inverse image of
$p'$. Observe that lines through $p$ and points in
$\rnn\backslash\{p\}$ either cross $\Gamma$ or parallel $\Gamma$,
and the formers are mapped into $\Gamma'$ and the latters are
contained in an $(n-1)$-di\-men\-sional hyperplane parallel to
$\Gamma$ and hence are mapped into an $(n-1)$-di\-men\-sional
hyperplane, say $\Lambda'$. Thus, $\rnn$ is mapped into the union of
$\Gamma'$ and $\Lambda'$, a contradiction.

The proof of Theorem \ref{Th:thm3} is completed by applying Theorem
\ref{Th:lw}.
\end{proof}

\begin{centering}\section{\!\!\!\!\!{. }
Concluding remarks}\label{S:simple}
\end{centering}

All  results mentioned in this paper belong to a young and active
geometrical discipline called ``characterizations of geometrical
mappings under mild hypotheses". The discipline started around 1950
with fundamental theorems  of A. D. Alexandrov on spacetime
transformations and causal automorphisms (see \cite{Bz}). Throughout
the conditions in these theorems, for examples, Theorems $A
\thicksim D$ and our main result, surjectivity, injectivity and
non-degenerate play inevitable  roles. In our result, we remove the
injectivity assumption but non-degenerate one is satisfied
automatically.  In a future paper \cite{Yao2}, the author even
further replaces the surjectivity assumption with the condition that
 every $r$-di\-men\-sional hyperplane  contains
at least $r+1$ image points. One may ask, what situation will be if
the surjectivity assumption on $f$ replaced by the injectivity one?

Actually, we can say nothing on $f$ because there exists a so-called
degenerate map $f:\,\hn\to\hn$ such that $f$ maps $\hn$ one-to-one
into some $\tilde{r}$-di\-men\-sional hyperplane. So, we need
another restriction on $f$ to guarantee that it is an automorphism.
Naturally, ``non-degenerate" is the first candidate. Maybe one
expects a theorem as follows:

\textit{If $f:\,\hn\to\hn$ is an injective map and maps an
$r$-di\-men\-sional hyperplane into an $r$-di\-men\-sional
hyperplane and if $f$ is non-degenerate, i.e., $f(\hn)$ cannot be
contained in an $r$-di\-men\-sional hyperplane, then $f$ is an
isometry.}

Unfortunately, recently in oral communication,
 a counterexample was  given by Li Baokui. We interpret it here.

\textit{Counterexample:} For convenience, let $n=2$ and  use the
semisphere $S$ in $\mathbb{R}^3$ as the model of $\mathbb{H}^2$,
namely, \begin{equation*}
 S=\{(x,y,z)\in\mathbb{R}^3:\,x^2+y^2+z^2=1,\,z>0\}.
 \end{equation*}
Then, all geodesics in $S$ are these semicircles perpendicular to
the $XOY$-plane. Project $S$ onto the unit disk $D$  in
$\mathbb{R}^2$:
\begin{equation*}
D=\{(x,y)\in\mathbb{R}^2:\,x^2+y^2<1\}.
 \end{equation*}
Say, let $P$ denote the projection:
\begin{equation*}
P(x,y,z)=(x,y).
 \end{equation*}
 It is easy to see that the
images of geodesics in $S$ under $P$ are these segments with ends at
the boundary of $D$. Let  $A$ be an affine transformation in
$\mathbb{R}^2$ with the form:
\begin{equation*}
A(x,y)=(ax,by),\;a,\,b\in (0,1).
 \end{equation*}
 Thus, the composition map $f=P^{-1}\circ
A\circ P$ maps $S$ into $S$.  It is evident that $f$ is injective
and non-degenerate. It maps a geodesic in $S$ into a geodesic and
the action of  $f$ on geodesics is shortening them and hence $f$ is
not an isometry.

Although   the exceptional phenomenon  occurs in $\hn$, we cannot
find such counterexample in $\rn$ or $\rnn$ so far. Whence, we end
 remarks with two open problems.

\begin{open} Suppose $f:\,\rn\to\rn$ is an injective map and  maps an
$r$-di\-men\-sional sphere into an $r$-di\-men\-sional sphere.  Can
we say that  $f$ is a M\"obius transformation if $f(\rn)$  cannot be
contained in an $r$-di\-men\-sional sphere?
\end{open}

\begin{open} Suppose $f:\,\rnn\to\rnn$ is an injective map and  maps an
$r$-di\-men\-sional hyperplane into an $r$-di\-men\-sional
hyperplane. Can we say that  $f$ is an affine transformation if
$f(\rnn)$  cannot be contained in an $r$-di\-men\-sional hyperplane?
\end{open}

\noindent\textbf{Acknowledgements. }The author would like to express
gratitude to the referee for his valuable comments. Thanks are also
due to Professor Xiantao Wang  for some idea of this paper inspired
by helpful communication with him.

\renewcommand\refname{\centerline{\Large{R}\normalsize{EFERENCES}}}

\end{document}